\documentclass[12pt]{amsart}

\usepackage{amsmath,amsthm,amsfonts,amssymb}

\newtheorem{theorem}{Theorem}[section]
\newtheorem{proposition}[theorem]{Proposition}
\newtheorem{lemma}[theorem]{Lemma}

\numberwithin{equation}{section}

\newcommand{\E}{\mathbb{E}}

\renewcommand{\P}{\mathbb{P}}

\renewcommand{\S}{\mathbb{S}}
\newcommand{\J}{\mathbb{J}}
\newcommand{\I}{\mathbb{I}}
\newcommand{\K}{\mathbb{K}}
\newcommand{\C}{\mathbb{C}}

\newcommand{\SC}{\mathcal{C}}

\newcommand{\SF}{\mathcal{F}}
\newcommand{\SG}{\mathcal{G}}

\newcommand{\SL}{\mathcal{L}}
\newcommand{\SM}{\mathcal{M}}
\newcommand{\SN}{\mathcal{N}}
\newcommand{\SO}{\mathcal{O}}
\newcommand{\SP}{\mathcal{P}}

\begin{document}

\baselineskip=15.5pt

\title[Poisson surfaces and algebraically completely integrable systems]{Poisson
surfaces and algebraically completely integrable systems}

\author[I. Biswas]{Indranil Biswas}

\address{School of Mathematics, Tata Institute of Fundamental Research,
Homi Bhabha Road, Bombay 400005, India}

\email{indranil@math.tifr.res.in}

\author[J. Hurtubise]{Jacques Hurtubise}

\address{Department of Mathematics, McGill University, Burnside
Hall, 805 Sherbrooke St. W., Montreal, Que. H3A 0B9, Canada}

\email{jacques.hurtubise@mcgill.ca}



\date{}

\begin{abstract} One can associate to many of the well known algebraically integrable systems of Jacobians (generalized Hitchin systems, Sklyanin) a ruled surface which encodes much of its geometry. If one looks at the classification of such surfaces, there is one case of a ruled surface that does not seem to be
covered. This is the case of projective bundle associated to the first jet bundle of a topologically nontrivial line bundle.
We give the integrable system corresponding to this surface; it turns out to be
a deformation of the Hitchin system.
\end{abstract}

\maketitle

\section{Introduction}

The geometry of algebraically integrable Hamiltonian systems, at least when the level 
sets of the Hamiltonians are Jacobians of curves, is intimately tied to the geometry 
of Poisson surfaces. Indeed, the well studied integrable systems of this type seem to 
come with a canonically associated Poisson surface, which encodes much of their 
geometry.

The canonical examples of this are the generalized Hitchin systems, for ${\rm GL}(n,\C)$. One fixes a compact Riemann surface $X$, and a positive divisor $D$ on $X$.
Let $K_X$ be the holomorphic cotangent bundle
of $X$. The phase space for the generalized Hitchin system consists of equivalence classes of pairs $(E,\phi)$ where $E$ is a rank $n$ holomorphic
vector bundle on $X$, and $\phi $ is a holomorphic
section of $End(E)\otimes K_X(D)$. The Hamiltonians are given by the coefficients of the spectral curve $S$ cut out in the total space $\K(D)$ of $K_X\otimes\SO(D)$ by
the equation ${\rm Det}(\phi-\eta\I)\,= \,0$, where $\eta $ is the tautological section of the pullback of $K_X\otimes\SO(D)$ to $\K(D)$. Fixing the Hamiltonians, one can consider the cokernel $\SF$ defined by 
$$
0\,\longrightarrow\, E\otimes K_X^\vee(-D)\, {\buildrel{\phi-\eta\I}\over{\longrightarrow}}
\, E\,\longrightarrow\, \SF\,\longrightarrow\, 0\, .
$$
When one is in a generic situation, say when $S$ is reduced smooth, the above cokernel $\SF$ is a line bundle
on $S$. Line bundles, of course, are parametrized by Jacobians, and one gets an integrable system of Jacobians, fibering over a base corresponding to a family of spectral curves in the surface $\K(D)$. When one specializes $X$ and the divisor $D$, this gives many of the classically studied integrable systems of Jacobians.

We compactify $\K(D)$ to $\overline{\K}(D)\, :=\,\P(\SO\oplus K(D))$, and note that this has a Poisson structure (a holomorphic section of $\bigwedge^2 T\overline{\K}(D)$) which vanishes along $D$, and vanishes to order two along the compactifying divisor
$\P(1,0)\,\subset\, \overline{\K}(D)$. This Poisson structure
encodes much of Poisson geometry of the Hitchin system. Indeed, normalizing the line bundle $\SF$ (tensoring by a fixed line bundle so that the result is of degree $g$, and so generically has a single non-zero section), the divisor $\sum_i(x_i,\eta_i)$ of that section provides Darboux coordinates which in fact mediate the separation of variables for the system in terms of Abelian integrals. (See \cite{AHH, Hu}.) More geometrically, this procedure defines a Poisson isomorphism between an open set of the Higgs moduli and an open set of the Hilbert scheme of points of the surface $\K(D)$.

When the genus of $X$ is $0$ or $1$, one also has the Sklyanin system, for which the 
phase space consists of pairs $(E\, ,\phi)$, where $E$ is a vector bundle as before, and 
$\phi$ is now a meromorphic automorphism of $E$; one again has a spectral curve, and 
this time the surface is $X\times \P^1$, but now with a Poisson structure which 
vanishes along two sections $X\,\longrightarrow \,X\times \P^1$ given by $0\, ,\infty 
\,\in\, \P^1$. One has Darboux coordinates as before \cite{HuMa}.

There is a general version of this picture, examined in \cite{Hu}, showing that in 
some sense the surface systems are the ones of minimal complexity. Indeed, given an 
integrable system $\J\,\longrightarrow\, B$ of Jacobians, with an associated family 
$\S\,\longrightarrow\, B$ of curves, one can take an Abel-Jacobi map $A\,:\,\S
\,\longrightarrow\,
\J$, and pull back the symplectic form $\Omega$. If this pullback $A^*\Omega$ is of 
minimal rank, one gets a symplectic surface by quotienting out the null foliation. The 
curves embed in this surface, and the surface again provides separating coordinates 
for the system. A parabolic bundle analog is studied in \cite{BGL}.

One is therefore interested in the Poisson surfaces $P$, with a view to seeing what integrable systems might correspond to it. Considering Kodaira dimension, and taking into account the fact that the Poisson tensor behaves well under blowing down, one reduces rapidly to the cases of $P$ an Abelian variety or a K3 surface, for which the surface is symplectic (we leave these cases aside), or $P$ a ruled surface. What is perhaps more surprising is that the set of possibilities for the latter is somewhat restricted. The classification was done by Bartocci and Macr\'i \cite{BaM}; rephrasing their result somewhat:

\begin{proposition} Let $P$ be a ruled surface $\P(V)$ with Poisson structure over a
Riemann surface $X$ of genus $g$, and let $\overline \beta\,:\,P\,\longrightarrow
\,X$ be the projection. Then one of the following two holds:
\begin{enumerate} 
\item The vector bundle $V$ on $X$ is a sum of line bundles, which one can normalize to
$L \oplus {\mathcal O} $ with the degree of $L$ positive or zero. 
\begin{itemize}
\item For $g\,>\,1$, or for ($g\,=\,1$ and $L$ of positive degree), $L\,=\,K_X(D)$, with
$D$ a positive or zero divisor. The divisors of the possible Poisson structures are of
the form $2E+ \pi^*(D')$, where $E$ is the divisor given by the projectivisation of the
image of $K_X(D)$ in $V$, and $D'$ is an effective divisor on $X$ linearly equivalent to $D$. 

\item For $g\,=\,1$, and the degree of $L$ zero, one can also get as divisor of the Poisson
structure the divisor $E+E'$, where $E$ and $E'$ are the divisors on $\P(V)$ corresponding
to the inclusions of $L$ and ${\SO}$ into their sum. 

\item For $g\,=\,0$, both possibilities for the divisor of the Poisson tensor can
occur: either the divisor is $2E+ \pi^*(D')$ or it is $E+E' + \pi^*(D'')$.
\end{itemize}

\item When $g\,\geq\, 1$ the bundle $V$ could also be the non-trivial extension
$$0\,\longrightarrow\, K_X\,\longrightarrow \,V\,\longrightarrow\,
\SO\,\longrightarrow\, 0\, .
$$
The vector bundle $V$ can be taken to be the tensor product $J^1(L)\otimes L^*$ of the 
one-jet bundle of a line bundle of non-zero degree, with the dual of that line bundle. 
The extension above is simply the one-jet sequence
$$0\,\longrightarrow\, K_X\otimes L\,\longrightarrow \,J^1(L)\,\longrightarrow\, 
L\,=\, J^0(L)\,\longrightarrow\, 0
$$
of $L$, tensored with the dual of 
$L$. In this case the divisor of the Poisson tensor is the divisor $2E$, where $E$ is 
given by the projectivisation of the inclusion $K_X\,\longrightarrow\, V$.
\end{enumerate}
\end{proposition}

The line bundle cases are essentially covered by the Hitchin or Sklyanin systems. One 
is then left with the question of trying to understand what, if anything, the surface 
corresponding to the non-trivial extension corresponds to, and this is the subject of 
this paper. We will find that a geometry very similar to that of the Hitchin systems 
holds; instead of Higgs fields $\phi$ which are endomorphisms taking values in the 
one-forms, we find ``shifted'' Higgs fields $\psi$ taking values in the connections on 
a fixed line bundle $L$. We will explore some of the properties of these 
$L$-connection-valued Higgs bundles.

The shifted moduli space mimics many of the aspects of the Hitchin moduli of Higgs 
bundles; it supports an integrable Hamiltonian system, for example, with spectral 
curves, and so on. The question arises as to whether it shares the more gauge 
theoretical properties associated to the Higgs bundles, for example a hyperK\"ahler 
structure in the parabolic case, or other complex structures tying one to some form of 
representation of the fundamental group. It is a question which we leave for another 
time, to focus more on the complex geometry. We note also that our 
$L$-connection-valued Higgs bundles can be defined for other structure groups, not 
only ${\rm GL}(n,{\mathbb C})$, though for these one no longer has a surface as the 
relevant geometric object (see \cite{HuMa2}).
 
\section{Jet bundles and connections}

Let $X$ be a compact connected Riemann surface. As before, the holomorphic cotangent bundle
of $X$ will be denoted by $K_X$. Fix a holomorphic line bundle
$L$ on $X$. Consider the short exact sequence of jet bundles
$$
0\,\longrightarrow\, K_X\otimes L\,\longrightarrow\, J^1(L)
\,\longrightarrow\, J^0(L) \,=\, L\,\longrightarrow\, 0
$$
associated to $L$. Tensoring it $L^*$, we get an exact sequence
\begin{equation}\label{e1}
0\,\longrightarrow\, K_X\,\longrightarrow\, {\mathcal V}_L\, :=\, J^1(L)
\otimes L^* \,\stackrel{\phi}{\longrightarrow}\, {\mathcal O}_X\,\longrightarrow\, 0\, .
\end{equation}
The above vector bundle ${\mathcal V}_L$ is the dual of the Atiyah bundle
$\text{At}(L)$ for $L$. Let $1_X$ be the section of ${\mathcal O}_X$ given by the
constant function $1$. Define the fiber bundle over $X$
\begin{equation}\label{e2}
{\mathcal C}_L\, :=\, \phi^{-1}(1_X(X))\,\subset\, {\mathcal V}_L\, .
\end{equation}
Holomorphic sections of ${\mathcal C}_L$ over an open subset $U$ of $X$ are
 the holomorphic connections on $L\vert_U$. From
\eqref{e1} it follows immediately that ${\mathcal C}_L$ is a torsor over $X$ for
$K_X$.

\begin{lemma}\label{lem1}
If ${\rm degree}(L)\,=\, 0$, then the fiber bundle ${\mathcal C}_L$ over $X$ is
holomorphically isomorphic to $K_X$.

If ${\rm degree}(L)\,\not=\, 0$, then ${\mathcal C}_L$ does not admit any
compact complex analytic subset of positive dimension.
\end{lemma}

\begin{proof} 
First assume that ${\rm degree}(L)\,=\, 0$. Then $L$ admits a holomorphic connection.
In fact, $L$ has a unique flat holomorphic connection whose monodromy lies in ${\rm U}(1)$.
Therefore, ${\mathcal C}_L$ admits a holomorphic section. So the $K_X$--torsor
${\mathcal C}_L$ is trivial; in particular, ${\mathcal C}_L$ is holomorphically
isomorphic to $K_X$.

Now let us consider the case of non-zero degree. Let
\begin{equation}\label{e3}
\beta\, :\, {\mathcal C}_L\,\longrightarrow\, X
\end{equation}
be the projection. Using the pullback operation, we have a natural injective homomorphism
$$
\beta^*J^1(L)\, \longrightarrow\, J^1(\beta^* L)\, .
$$
Tensoring this with $\text{Id}_{\beta^{^*} L^{^*}}$, the following homomorphism
$$
\alpha\, :\, \beta^*(J^1(L)\otimes L^*) \, \longrightarrow\, J^1(\beta^* L)
\otimes (\beta^* L)^*
$$
is obtained. The vector bundle $\beta^*(J^1(L)\otimes L^*)$ over ${\mathcal C}_L$ has
a tautological section. This tautological section will be denoted by $s$. Now the section
$$
\alpha\circ s \, :\, {\mathcal C}_L\,\longrightarrow\,J^1(\beta^* L)
\otimes (\beta^* L)^*
$$
defines a holomorphic connection on the line bundle $\beta^* L$. Let
\begin{equation}\label{f1}
{\mathcal D}_L
\end{equation}
denote this tautological holomorphic connection on $\beta^* L$.

Let
$$
\xi\, :\, Y\, \longrightarrow\, {\mathcal C}_L
$$
be a nonconstant holomorphic map from a compact connected Riemann surface $Y$.
Then the pullback $\xi^*{\mathcal D}_L$ is a holomorphic connection on the line bundle
$\xi^*\beta^* L\,=\, (\beta\circ\xi)^*L$. This implies that
\begin{equation}\label{e4}
\text{degree}(\beta\circ\xi)\cdot \text{degree}(L)\,=\, 
\text{degree}((\beta\circ\xi)^*L)\,=\, 0\, .
\end{equation}

Since the fibers of $\beta$ are affine spaces, $Y$ is not contained in
some fiber of $\beta$. Therefore, $\text{degree}(\beta\circ\xi)\, >\, 0$.
But this contradicts \eqref{e4}. Therefore, we
conclude that ${\mathcal C}_L$ does not admit any compact
complex analytic subset of positive dimension if $\text{degree}(L)\,\not=\,0$.
\end{proof}

Since $\dim H^1(X,\, K_X)\,=\, 1$, given any two nontrivial $K_X$--torsors
$A$ and $B$ on $X$, there is a holomorphic isomorphism of fiber bundles over
the identity map of $X$
$$
t\, :\, A\,\longrightarrow\, B
$$
and a constant $c\,\in\, {\mathbb C}\setminus\{0\}$ such that
$t(v+\omega)\,=\, t(v)+c\cdot\omega$ for all $v\,\in\, A_x$,
$\omega\,\in\, (K_X)_x$ and $x\,\in\, X$. In particular, the
two fiber bundles $A$ and $B$ are holomorphically isomorphic.

Let
\begin{equation}\label{f2}
\Omega\,:=\, {\mathcal K}({\mathcal D}_L)\,\in\, H^0({\mathcal C}_L,\,
\Omega^2_{{\mathcal C}_L})
\end{equation}
be the curvature of the connection ${\mathcal D}_L$ in \eqref{f1}. This
$\Omega$ is a holomorphic symplectic form on ${\mathcal C}_L$.
If we fix a trivialization of $L$ over an open subset $U$, then $\beta^{-1}(U)$
gets identified with the total space of $K_U$ (the trivialization produces a
holomorphic connection on $L\vert_U$). This identification takes the
symplectic form $\Omega\vert_{\beta^{-1}(U)}$ to the standard Liouville
symplectic form on $K_U$.

Note that
$$
{\mathcal C}_L\,\subset\, P = {\mathbb P}(J^1(L)\otimes L^*)\,=\,
{\mathbb P}(J^1(L))\, .
$$
The divisor ${\mathbb P}(J^1(L))\setminus {\mathcal C}_L$ will be
denoted by $D_\infty$. The symplectic form $\Omega$ in \eqref{f2}
has a pole of order two at $D_\infty$; dually the Poisson structure has a
double zero there. This follows from the above local
identification of $\Omega$ with the Liouville symplectic form.

\section{Geometry of sheaves on a Poisson surface}.
 
Following on work of Mukai \cite{Mu}, in the symplectic case, and Bottacin \cite{Bo}, 
in the Poisson case, a Poisson structure on a complex surface $P$ induces a Poisson 
structure on the various spaces of sheaves over the surface. The cases which will 
interest us are moduli spaces $\SM$ of sheaves $\SF$ with support of pure dimension 
one, and fixed numerical invariants. We will most of the time restrict to a subspace 
$\SM_0$ of pairs $(S\, ,\SF)$ of line bundles $\SF$ of fixed degree over a reduced 
curve $S$. The relevant geometry has been extensively covered for cases very similar 
to this one, in various places (see, e.g. \cite{HuMa}), so our presentation will be 
somewhat brief.
 
 Let $$\theta\, \in\, H^0(P,\, K_P^*)$$ be the Poisson structure on the smooth projective surface $P$; the divisor of $\theta$ is $D= 2D_\infty$. The Poisson structure on the moduli of sheaves is defined as follows, following \cite{Bo}: one has that the first order deformations of $\SF$ are given by the global Ext-group
 $\E xt^1(\SF,\, \SF)$, dually, one has the cotangent space $\E xt^1(\SF,\, \SF\otimes K_P)$. One has natural maps
 $$\E xt^1(\SF,\,\SF\otimes K_P)\otimes \E xt^1(\SF,\,\SF\otimes K_P)\longrightarrow \E xt^2(\SF,\,\SF\otimes K_P^2)
$$
$$
\buildrel{\theta}\over{\longrightarrow} \E xt^2(\SF,\SF\otimes K_P )\longrightarrow \C\, .$$
 The first arrow is the standard pairing; the second is multiplication by the Poisson structure; the third is Grothendieck-Serre duality. This defines the Poisson tensor on the moduli, as a section of the second exterior power of the tangent bundle.
 Alternately, one can define the Poisson tensor as a (skew) map from the cotangent space to the tangent space
 $$\E xt^1(\SF,\, \SF\otimes K_P) \, \longrightarrow\, \E xt^1(\SF,\,\SF)\, , $$
 given here by multiplication by the Poisson tensor. This indeed defines a Poisson structure, i.e., satisfies the correct integrability conditions. The latter can be proved directly; it also follows from a local isomorphism with the Hilbert scheme of points on the surface, as explained below.
 
 Let us restrict to the case of sheaves supported on curves; more specifically, to the moduli of sheaves whose generic element $\SF$ is a line bundle over a smooth curve $S$. Then one can make the definition more explicit, as in \cite{HuMa}, or indeed, several other references. One can take an extension of $\SF$ to a neighborhood of $S$, obtaining a resolution by locally free rank one sheaves on the neighborhood:
 $$0\longrightarrow \SF_1\longrightarrow \SF_0 \longrightarrow \SF\longrightarrow 0\, .$$
Over $S$, the sheaf $\SF_0$ is isomorphic to $\SF$, and $\SF_1$ is isomorphic to the tensor product of $\SF$ with the conormal bundle $N_S^*$. Computing global Ext, one then has for the tangent space
$$ 0\longrightarrow H^1(S,\SO) \longrightarrow \E xt^1(\SF,\SF) \longrightarrow H^0(S, N_S) = H^0(S, K_S(D))\longrightarrow 0\, .$$
The isomorphism of $N_S$ and $K_S(D)$ uses the Poisson structure and the Poincar\'e adjunction formula.
For the cotangent space, one has:
$$ 0\longrightarrow H^1(S,\,\SO(-D)) \longrightarrow \E xt^1(\SF,\,\SF\otimes K_D) \longrightarrow H^0(S,\, N_S) = H^0(S,\, K_S)\longrightarrow 0\, .$$
The Poisson structure, thought of as a homomorphism $\E xt^1(\SF,\,\SF\otimes K_D)\,\longrightarrow \,\E xt^1(\SF,\,\SF)$, is given by multiplication by the Poisson tensor $\theta$; it maps the cotangent sequence above to the tangent sequence. As in \cite{HuMa}, one has:

\begin{proposition} 
For our Poisson variety of pairs $(S,\SF)$, the symplectic leaves $\SL$ are given by asking that the intersection of the curve $S$ with $D= 2D_\infty$ remain constant, so that, to first order, deformations of the curve live in
$$H^0(S,\, N_S(-D)) \,=\, H^0(S,\, K_S)\,\subset\, H^0(S,\, N_S) \,= \,H^0(S, \,K_S(D))\, .$$
On the symplectic leaves, one then has for the tangent space (and so the cotangent) space:
$$
0\,\longrightarrow\, H^1(S,\SO)\,\longrightarrow\,T\SL\,\longrightarrow\, H^0(S, K_S)
\,\longrightarrow\, 0\, .$$
The subspace $H^1(S,\,\SO)$ is isotropic with respect to the symplectic form; the
map $$(S,\,\SF)\,\longrightarrow \,S$$ defines a
completely integrable system on $\SM_0$.
\end{proposition}

One can, again following \cite{HuMa}, define a local isomorphism of the symplectic leaves $\SL$ with the Hilbert scheme of points on the curve. The idea is to choose a fixed line bundle $\SG$ on a neighborhood of a curve $S$ in our family, so that the sheaves $\SF\otimes \SG$ are of degree $g=$genus($S$), and so have generically one non-zero section (up to scale) with divisor a sum $\sum_ip_i$ of points on $S$. These points can be thought of as points of $P$, cut out by the defining equation of $S$ and further equations, and so the divisor $\sum_ip_i$ on the curve defines an element
of the Hilbert scheme ${\rm Hilb}^g(P)$. This space has an obvious Poisson structure, induced from that on $P$. One has, as in \cite{HuMa}:

\begin{proposition}
The (locally defined) map $\SL \,\longrightarrow\, {\rm Hilb}^g(P)$ constructed above is Poisson.
\end{proposition}

This shows, incidentally, that the section of $\Lambda^2(T\SM_0)$ that we have defined above on $\SM_0$ is indeed a Poisson structure, (in the sense of the corresponding symplectic forms on the leaves being closed), since the structure on ${\rm Hilb}^g(P)$ (a symmetric product on a Zariski open dense set) is obviously Poisson.

\section{$L$-connection-valued Higgs bundles}

We thus have a Poisson manifold, obtained as a moduli space of sheaves over $P$. The question is what does this correspond to over $X$. Not surprisingly, as $P$ is a
deformation of $\P(\SO\oplus K_X)$, the answer turns out to be a deformation of the moduli space of Higgs bundles over $X$. We will define a moduli space of pairs $(E\, ,\psi)$
with $E$ a holomorphic vector bundle over $X$, and $\psi$ a section of $End(E)$ with
values in the connections on $L$, i.e., a $\SC_L$-valued section. As we have seen, holomorphic sections of $\SC_L$ are rather rare, so we will allow poles.

\subsection{Polar parts}

Fix an effective divisor $C$ over $X$. Let
\begin{equation}\label{e7}
0\,\longrightarrow\, K_X\otimes {\mathcal O}_X(C)\,\longrightarrow\,
{\mathcal V}_L\otimes {\mathcal O}_X(C)
\,\stackrel{\widetilde\phi}{\longrightarrow}\, {\mathcal O}_X(C)\,\longrightarrow\, 0
\end{equation}
be the exact sequence obtained by tensoring \eqref{e1} with the holomorphic
line bundle ${\mathcal O}_X(C)$. This produces the short exact sequence
\begin{equation}\label{e5}
0\,\longrightarrow\, K_X\otimes {\mathcal O}_X(C)\,\longrightarrow\,
\widetilde{\mathcal V}_L\,:=\, {\widetilde\phi}^{-1}({\mathcal O}_X)\,
\stackrel{\widetilde\phi}{\longrightarrow}\, {\mathcal O}_X\,\longrightarrow\, 0
\end{equation}
using the inclusion of ${\mathcal O}_X$ in ${\mathcal O}_X(C)$ (recall that the divisor
$C$ is effective). We have a surjective homomorphism
$$
\widetilde{\mathcal V}_L \,\longrightarrow\, K_X(C)_C\, , 
$$
where $K_X(C)_C\,=\, K_X(C)\vert_C $ is the restriction to $C$ of the line bundle $K_X(C)$.

Indeed, the natural inclusions of sheaves
$$
K_X\,\hookrightarrow\, K_X(C) \ ~ \text{ and }~ \
{\mathcal V}_L\,\hookrightarrow\, \widetilde{\mathcal V}_L
$$
fit together in the following commutative diagram
$$
\begin{matrix}
&& 0 && 0\\
&& \Big\downarrow && \Big\downarrow\\
0& \longrightarrow & K_X & \longrightarrow & {\mathcal V}_L &
\stackrel{\phi}{\longrightarrow} & {\mathcal O}_X & \longrightarrow & 0\\
&& \Big\downarrow && \Big\downarrow && \Vert \\
0& \longrightarrow & K_X(C) & \longrightarrow &
\widetilde{\mathcal V}_L &
\stackrel{\widetilde\phi}{\longrightarrow} & {\mathcal O}_X & \longrightarrow & 0\\
&& \Big\downarrow && ~\ \Big\downarrow \SP\SP \\
&& K_X(C)_C &= &
K_X(C)_C\\
&& \Big\downarrow && \Big\downarrow\\
&& 0 && 0
\end{matrix} \quad .
$$
For example, if $C$ is reduced, then at each point $x_0$ of $C$, using the Poincar\'e adjunction formula, the fiber
${\mathcal O}_X(x_0)_{x_0}$ is identified with $T_{x_0}X$. Therefore,
$$
K_X(x_0)_{x_0}\,=\, \mathbb C\, .
$$
We refer to the homomorphism $\SP\SP: \widetilde{\mathcal V}_L\,\longrightarrow\, K_X(C)_C$ in the commutative diagram as the \textit{polar part homomorphism}.

\subsection{Definition of $L$-connection-valued Higgs bundles}

Let $L$ be a fixed line bundle over $X$. As before, fix an effective divisor $C$ over $X$. Fix a conjugacy class $A$ in $gl(n,\C)\otimes_\C K_X(C)_C$ under the action of the maps of $C$ into 
${\rm GL}(n,\C)$. For $C$ reduced, this amounts to fixing a ${\rm GL}(n,\C)$--conjugacy class in $gl(n,\C)\otimes_\C K_X(C)_x$ for every $x\,\in\, C$.

An \textit{$L$-connection-valued Higgs bundle with poles at $C$, and polar part in $A$} is
a pair $(E\, ,\psi)$, where $E$ is a holomorphic vector bundle on $X$ and
$$
\psi\,\in\, H^0(X,\,End(E)\otimes \widetilde{\mathcal V}_L)
$$
(see \eqref{e5}) such that the following two conditions hold:
\begin{enumerate}
\item the image $(\text{Id}\otimes\widetilde{\phi})(\psi)\, \in\, H^0(X,\, 
End(E))$, where $\widetilde{\phi}$ is the homomorphism in \eqref{e5}, coincides
with the identity endomorphism of $E$, and

\item the polar part of $\psi$ lies in $A$.
\end{enumerate}

Let $(E\, ,\psi)$ be such a $L$-connection-valued Higgs bundle. Using the Lie algebra
structure of the fibers of $End(E)$ together with the natural projection
$$
\widetilde{\mathcal V}_L\otimes \widetilde{\mathcal V}_L\,\longrightarrow\,
\bigwedge\nolimits^2 \widetilde{\mathcal V}_L\, ,
$$
we get a homomorphism
$$
(End(E)\otimes \widetilde{\mathcal V}_L)\otimes (End(E)\otimes \widetilde{\mathcal V}_L)
\,\longrightarrow\,End(E)\otimes \bigwedge\nolimits^2 \widetilde{\mathcal V}_L\, .
$$
The image of any $\alpha_1\otimes\alpha_2$ by this homomorphism will be denoted
by $\alpha_1\bigwedge\alpha_2$. Since the identity map of $E$ commutes with every
endomorphism of $E$, from the given condition that
$(\text{Id}\otimes\widetilde{\phi})(\psi)\, =\, \text{Id}_E$ it follows immediately
that
\begin{equation}\label{e8}
\psi\bigwedge\psi\,=\, 0\, .
\end{equation}

\subsection{Moduli of $L$-connection-valued Higgs bundles}

A $L$-connection-valued Higgs bundle $(E\, ,\theta)$ with poles at $C$, and polar part in $A$ reduces to a parabolic Higgs bundle when $L$ has degree zero, and one can use this to define stability and obtain a moduli space. It is unclear what the correct definition of stability should be when the degree of $L$ is non-zero; we are here, however, interested in the local geometry, and so we will simply restrict our attention to $L$-connection-valued Higgs bundles $(E\, ,\psi)$
such that the underlying vector bundle $E$ is stable. For this reason we also assume
that
$$
\text{genus}(X)\, \geq\, 2\, .
$$

Henceforth, we will consider only the space $\SN_0$ consisting of those
$L$-connection-valued Higgs bundles $(E\, ,\psi)$
for which the underlying vector bundle $E$ is stable. The moduli space
of stable vector bundles exists as a smooth quasiprojective variety after we fix
the rank and the degree. For a stable vector bundle $E$ on $X$, we have
$$
H^1(X,\, End(E)\otimes K_X)\,=\, H^0(X,\, End(E))^*\, =\, \mathbb C\, .
$$
The moduli space $\SN_0$ of $L$-connection-valued Higgs bundles of fixed rank and degree is a holomorphic
fiber bundle over the moduli space of stable vector bundles of that rank and degree.

\subsection{Poisson and symplectic structures}

We begin by writing out a deformation complex for our Higgs bundles, at a point $(E\, ,
\psi)$. The way to do this is by now fairly standard (\cite{BiRa}, \cite{Ma}).
Indeed, consider first the deformations $\psi(\epsilon)$ of $\psi$. First of all, we
note that the projection of $\psi(\epsilon)$ to $End(E)$ is a constant
element (the identity), so that the $\epsilon$-derivative $\psi'$ of $\psi$ at
$\epsilon\,=\,0$ lives in $End(E)\otimes K_X(C)$. Furthermore, the constraint that
the polar part of $\psi$ lie in a fixed conjugacy class means that the polar part of
$\psi'$ is the polar part of an element $[\alpha\, ,\psi]$, where $\alpha$ is a local
holomorphic section of $End(E)$. (Note that the bracket of $\psi$ with any section of
$End(E)$ automatically takes values in $End(E)\otimes K_X(C)$). The first order
deformations are thus locally of the form $\psi(\epsilon)\,= \,\psi +
\epsilon(\psi_{reg}' + [\alpha,\psi])$, where $\alpha, \psi'_{reg}$ are holomorphic. Let
$End(E)_\psi\otimes K_X\,\subset\, End(E)\otimes K_X(C)$ denote the subsheaf of such
elements $\psi_{reg}' + [\alpha,\psi]$. For example, near a (simple) pole $p\in C$, if
the polar part's only non-zero entry is in the $(1,1)$ position (choosing a
trivialization of $E$), the sheaf $End(E)_\psi$ near $p$ would be the sheaf of
sections of $End(E)$ with simple poles allowed at $p$ in the $(1,j)$ and $(j,1)$
entries, $j\neq 1$.

The derivative $\psi'$ of our local deformations of $\psi$ takes values in
$End(E)_\psi$; bundles, on the other hand, have deformations living in $H^1(X,\,
End(E))$. The two deformations fit together, so that the tangent space at $(E\, ,\psi)$
is the first
hypercohomology of the complex 
\begin{equation}\label{tangent}
End(E)\stackrel{[\psi,\cdot]}{\longrightarrow}End(E)_\psi\otimes K_X,
\end{equation}
giving a sequence
$$H^0(X, End(E)_\psi\otimes K_X) \longrightarrow T\SN_0 \longrightarrow H^1(X, End(E))$$
Dually, let $End(E)_\psi^0$ be the kernel of $End(E)\stackrel{[\psi,\cdot]}{\longrightarrow}(End(E)_\psi\otimes K_X)_C$, i.e., the subsheaf of sections of $End(E)$ which remain holomorphic after bracketing with $\psi$. This sheaf $End(E)_\psi^0$ is the dual bundle to $End(E)_\psi$. One has that the cotangent space is the first hypercohomology of the complex
\begin{equation}\label{cotangent}
End(E)_\psi^0\stackrel{[\psi,\cdot]}{\longrightarrow}End(E)\otimes K_X
\end{equation}
The cotangent complex embeds naturally into the tangent complex, and on the level of hypercohomology, this induces the Poisson structure, thought of as a
homomorphism $T^*\longrightarrow T$, as in \cite{HuMa}. One would, of course, like to see that the structure is symplectic. To do this, one can consider it as a reduction of the symplectic structure on a larger moduli space of triples $(E, tr,\psi)$, where $E$ is a bundle as before, $tr$ is a trivialization of $E$ over $C$, and the conjugacy class of the polar part of $\psi$ at $C$ is arbitrary. The deformation complex for the tangent space, and the cotangent space, now becomes
$$End(E)(-C)\stackrel{[\psi,\cdot]}{\longrightarrow}End(E)\otimes K_X(C)$$
The symplectic structure is induced by the identity map, which of course induces an isomorphism. The space has a Hamiltonian action by the maps of $C$ into
${\rm GL}(n,\C)$, whose moment map is the evaluation at $C$ of $\psi$ in the trivialization $tr$. Reducing, we get our moduli space. This reduction basically just copies the generalized Hitchin case; see \cite{Ma}. 
 
\subsection{Spectral data for a $L$-connection-valued Higgs bundle}

One has the bundle of affine lines $\SC_L \subset {\mathcal V}_L$ as the elements mapping to $1$ in $\SO$. Let
\begin{equation}\label{f}
f\, :\, {\mathcal V}_L\,\longrightarrow\, X
\end{equation}
be the natural projection. The pulled back vector bundle $f^*\widetilde{\mathcal V}_L$
has a tautological section; this tautological section will be denoted by $\eta$; we restrict it to $\SC_L$. If one compactifies from $\SC_L$ to $P$, adding in the divisor $D_\infty$ , the tautological section $\eta$ has a single pole along infinity in $P$. As above, let $\overline{\beta}: P\longrightarrow X $ denote the projection from $P$.

Let $(E\, ,\psi)$ be a $L$-connection-valued Higgs bundle of rank $n$. Let $E_0$ be
the subsheaf of sections $s$ of $E$ such that $\psi(s)$ is holomorphic as a section of
$E\otimes{\mathcal V}_L$, i.e., has no poles at $C$. Noting that on $\SC_L$, the difference $\overline{\beta}^*\psi- \text{Id}_{\overline{\beta}^*E}\otimes\eta$ takes values in $\overline{\beta}^*(K_X(C))$, we have the homomorphism over $P$
\begin{equation}\label{e9}
0\,\longrightarrow \,\overline{\beta}^*(E_0\otimes K^*_X )(-D_\infty)\,
\stackrel{\overline{\beta}^*\psi- \text{Id}_{\overline{\beta}^*E}\otimes\eta}{\longrightarrow}\,
\overline{\beta}^*E\, \longrightarrow \SF\longrightarrow 0
\end{equation}
where $\eta$ is the tautological section defined above; this sequence defines a quotient sheaf $\SF = \SF(E,\psi)$. Let
$$
\bigwedge\nolimits^n (\overline{\beta}^*\psi- \text{Id}_{\overline{\beta}^*E}\otimes\eta)\, :\,
\bigwedge\nolimits^n(\overline{\beta}^*(E_0\otimes K^*_X)(-D_\infty))\,
\longrightarrow\, \bigwedge\nolimits^n (\overline{\beta}^*E)
$$
be the corresponding homomorphism between the exterior products, where $r$ is
the rank of $E$. Let
$$
f\, \in\, H^0(P, 
(\bigwedge\nolimits^n (\overline{\beta}^*\psi- \text{Id}_{\overline{\beta}^*E}\otimes\eta)^*\otimes
(\bigwedge\nolimits^n(\overline{\beta}^*E))
$$
be the section given by this homomorphism. The subscheme cut out by the vanishing of $f$ is 
the \textit{spectral curve} $S(E\, ,\psi)$. It is the support of the sheaf $\SF$.

\begin{proposition} The pair $(S, \SF)$ encodes the pair $(E,\psi)$.
\end{proposition}

\begin{proof}
The push-down $\overline{\beta}_*(\SF)$ is isomorphic to $E$; this follows by pushing
down the sequence (\ref{e9}), and noting that all the direct images of $\overline{\beta}^*(E\otimes K^*_X(-C))(-D_\infty)$ vanish. The section $\psi$ is then the pushdown of $\eta: \SF\longrightarrow \SF\otimes\overline{\beta}^*(\widetilde{\mathcal V}_L)$.
\end{proof}

We note that $\psi$ considered as a section of $End(E)\otimes {\mathcal V}_L$ has poles
at $C$, with polar parts lying in $End(E)\otimes K_X$. If the polar part of $\psi$ at $C$
has rank greater than one, so that the ``eigenspace with infinite eigenvalue'' at $C$
also has rank greater than one, this means that the spectral curve has several branches converging on infinity as one goes to $C$. We therefore put in some simplifying assumptions from now on, namely that the polar parts should be of rank one; we also ask that
$C$ be reduced, so that the poles are simple, and that the spectral curve be smooth. With these restrictions, the conjugacy class of the polar part of the connection is given simply by the residue of the trace at points $p_i$ of $C$:
$$res_i\circ tr \,:\, H^0(X, End(E)\otimes K_X(C))\longrightarrow \C\, .$$
If this is non-zero, there is only one branch of the spectral curve intersecting
$D_\infty$ transversely at $p_i$; if it is zero, there will be a simple branch point.. 

It is useful to expand the $End(E)\otimes K_X $-component $\hat \psi$ of $\psi$ in an adapted trivialization of the rank $n$ bundle $E$ near a point $p_i$ of $C$. Let $x$ be a coordinate on $X$ with $p_i$ corresponding to $x=0$, and $\mu = \eta^{-1}$ be the $\P^1$-coordinate, so that $D_\infty$ is given by $\mu = 0$
When $res_i\circ tr(\psi)\,\neq\, 0$ (``first case''), one can normalize to the form
\begin{equation}\label{nform1}\hat\psi =\begin{pmatrix} a_{-1} x^{-1} + a_0 +a_1x+...& 0\\0 & A_0 +A_1x+..\end{pmatrix},\end{equation}
where $a_i$ are constants, and $A_i$ are $(n-1)\times (n-1)$ matrices. One has $res_i\circ
tr(\psi)\,=\, a_{-1}$.

If $res_i\circ tr(\psi)\,=\,0$, (``second case'') one has the normal form
\begin{equation}\label{nform2}\hat\psi =\begin{pmatrix} 0& a_0 +a_1x+...& 0\\x^{-1}& b_0+ b_1x+... & 0 \\ 0&0 & A_0 +A_1x+..\end{pmatrix},\end{equation}
with $a_i, b_i$ constants, $a_0\neq 0$ and $A_i$ matrices of size $(n-2)\times (n-2)$. 

In both cases, the intersection of the spectral curve with $\mu^2=0$ is given by $(res_i\circ tr(\psi)) \mu - x\,= \,0$: as noted above, the first case corresponds to the point at infinity being a regular point for the projection to $X$, while the second case is a simple branch point.
Also, if one looks at the intersection of the spectral curve with $D_\infty$, one has that this trace residue is in essence the derivative of the projection of the spectral curve at the points of intersection of the spectral curve with $D_\infty$ ; in other words, fixing the conjugacy class of the rank one residue amounts to fixing the intersection of the spectral curve with the first formal neighborhood $2D_\infty$ of $D_\infty$. If one has fixed this conjugacy class, the family of spectral curves intersects the divisor of the Poisson structure on $P$ in a fixed locus, and so the family of $(S,\SF)$ corresponding to the $(E,\psi)$ lie in a fixed symplectic leaf of the family of sheaves on $P$.

Let $\SN_{0,reg}$ be the subset of the moduli space $\SN_0$ of elements $(E,\psi)$ for which the spectral curves are smooth and reduced.

\begin{theorem}\label{thmlas}
Let $C$ be reduced, and let the conjugacy classes $A_i$ at each point of $C$ be of rank one. The map $\S: \SN_{0,reg}\,\longrightarrow\,
\SM_0$ which associates to $(E,\psi)$ the spectral data $(S,\SF)$ maps to a fixed symplectic leaf $\SL$ of $\SM_0$. The map $\S$ is symplectic.
\end{theorem}
 
 We first prove a lemma. Noting that we have already that the pushdown to $X$ from $P$ of $\SF$ is $E$, we obtain:

 \begin{lemma}\label{lemlas}
The pushdown to $X$ from $P$ of $\SF(-D_\infty)$ is the subsheaf $E_0$ of sections $s$ of $E$ such that $\psi(s)$ is finite in $E$; that of $\SF(-2D_\infty)$ is the subsheaf $E_{00}$ of sections $s$ of $E$ such that $\psi(s)$ lies in $E_0$;
 that of $\SF(D_\infty)$ is the image $E_\psi \,:=\, \psi(E)$ in $E(C)$.
 
 The sheaf $End_\psi$ defined above can be identified as the subsheaf of $E_0^*\otimes E_\psi$ with no second order poles at $C$, and with first order polar part having trace zero.
 \end{lemma}

\begin{proof}[Proof of Lemma \ref{lemlas}]
The first statements follow from the exact sequences
$$0\longrightarrow \SF((n-1)D_\infty)\longrightarrow \SF(nD_\infty)\longrightarrow \SF(nD_\infty)|_{D_\infty}\longrightarrow 0\, .$$
One can also see this using the explicit normal forms; for example, in the first case, $E_0$ is the subsheaf of sections of $E$ with first entry vanishing at $x=0$, and $E_{00}$ is the subsheaf of sections whose first entry vanishes to order two. The statement about $End(E)_\psi$ can similarly be seen in local coordinates: in our first case, $E_0^*\otimes E_\psi$ consist of matrices whose $(1,1)$ term has a pole of order two, and whose $(1,k)$ and $(k,1)$ terms have simple poles, with the other entries being holomorphic, while $End(E)_\psi$ consists of matrices whose only poles are simple, and occur in the $(1,k)$ and $(k,1)$ entries for $k\neq1$. The second case can be analyzed in the same way.
\end{proof}
 
\begin{proof}[Proof of Theorem \ref{thmlas}] We have already seen that the image of $\S$ lies in a symplectic leaf. We now want to see that the map is symplectic.

The sequence \eqref{e9} gives us a resolution of $\SF$. Taking duals and tensoring with $\SF$, one finds that the tangent space to $\SM_0$ at $(S,\SF)$ is the first hypercohomology of
$$\overline{\beta}^*E^*\otimes \SF
\,\longrightarrow\,\overline{\beta}^*(E_0^*\otimes K_X )(D_\infty)\otimes \SF\, ,$$
while the cotangent space is the first hypercohomology of 
$$\overline{\beta}^*E^*\otimes \SF(-2D_\infty)\,\longrightarrow\,
\overline{\beta}^*(E_0^*\otimes K _X )(-D_\infty)\otimes \SF\, ,$$
recalling that the canonical bundle of $P$ is $\SO(-2D_\infty)$. The Poisson tensor, as a homomorphism from the cotangent space to the tangent space, is given simply by multiplication by the Poisson tensor on $P$ on these complexes. We can push them down to $X$, so that our tangent space is now the first hypercohomology of 
\begin{equation} \label{tgt} End(E){ \stackrel{[\psi,\cdot]} {\longrightarrow } }(E_0^*\otimes E_\psi) \otimes K_X\end{equation}
with that of the cotangent space as the first hypercohomology of 
\begin{equation} \label{cotgt}E^*\otimes E_{00} \stackrel {[\psi,\cdot]} {\longrightarrow } End(E_0) \otimes K_X = End(E)\otimes K_X\end{equation}

The leaf $\SL$ is characterized as having support with fixed intersection with $2D_\infty$. Referring to its defining
complex (\ref{e9}), one is interested in deformations $\psi'$ of the map $\psi$ which, when mapped to $\SF$, vanish over the intersection of the spectral curve with $2D_\infty$. Computing from the normal forms (\ref{nform1}),(\ref{nform2}), we find that this is effected by replacing the sheaf $E_0^*\otimes E_\psi$ by its subsheaf $End(E)_\psi $ in the sequence (\ref{tgt}). One then wants for the tangents to the leaf
the first hypercohomology of the subcomplex of \eqref{tgt}
\begin{equation} \label{tgt2} End(E) \stackrel {[\psi,\cdot]} {\longrightarrow } End(E)_\psi \otimes K_X\, .\end{equation}
Dually, for the cotangent space, one has the first hypercohomology of the cotangent complex
\begin{equation} \label{cotgt2}
End(E)_\psi^0\,\stackrel{[\psi,\cdot]}{\longrightarrow}\, End(E)\otimes K_X\, .
\end{equation}
These are precisely the deformation complexes for our $L$-connection-valued Higgs fields. The Poisson structure on $P$ can be thought of as the inclusion $\SO(-2D_\infty) \longrightarrow \SO$; after pushdown, this gets translated simply into the natural inclusion of (\ref{cotgt2})
 into (\ref{tgt2}). But this is the definition of the Poisson structure for the $L$-connection-valued Higgs fields, and so we are done.
\end{proof}

\section*{Acknowledgements}

We would like to thank the referee for useful comments.
We thank the Institute of Mathematical Sciences at Chennai, where the work began, for 
its hospitality. The first-named author acknowledges the support of the J. C. Bose 
Fellowship. The second was supported by NSERC and the FRQNT.

\end{document}